\documentclass[11pt]{amsart}
\usepackage[margin=1in]{geometry}
\usepackage{xcolor}
\usepackage[colorlinks=true]{hyperref}
\numberwithin{equation}{section}
\newtheorem{thm}{Theorem}[section]
\newtheorem{prop}[thm]{Proposition}
\newtheorem{lemma}[thm]{Lemma}

\newtheorem{defn}[thm]{Definition}

\newcommand{\wot}{W^{1,2}\left(\Omega\right)}
\newcommand{\wtt}{W^{2,2}\left(\Omega\right)}
\newcommand{\po}{\partial\Omega}
\newcommand{\iot}{\int_{\Omega_T}}
\newcommand{\ot}{\Omega_T}
\begin{document}
	\title[Fourth-order exponential pde's]{Existence Theorems For a Fourth-Order Exponential PDE Related to Crystal Surface Growth}
	\author{Brock C. Price and Xiangsheng Xu}\thanks
	{Department of Mathematics and Statistics, Mississippi State
		University, Mississippi State, MS 39762.
		{\it Email}: bcp193@msstate.edu (Brock C. Price); xxu@math.msstate.edu (Xiangsheng Xu).}
	\begin{abstract} 	In this article we prove the global existence of a unique strong solution to the initial boundary-value problem for a fourth-order exponential PDE. The equation we study was originally proposed to study the evolution of crystal surfaces, and was derived by applying a nonstandard scaling regime to a microscopic Markov jump process with Metropolis rates. Our investigation here finds that compared to the PDE's which use Arhenious rates, (and also have a fourth order exponential nonlinearity) the hyperbolic sine nonlinearity in our equation can offer much better control over the exponent term even in high dimensions. 
	\end{abstract}
	\maketitle
	
	\section{Introduction}\label{intro}
	
	Let $\Omega$ be a bounded domain in $\mathbb{R}^N$ with smooth boundary $\partial \Omega$ and $T >0$. We consider the initial boundary-value problem,
	\begin{align}
		\frac{\partial u}{\partial t} -\Delta\sinh(-\Delta u) = 0 \;\; \mbox{ in } \; \Omega_T \equiv \Omega \times \left(0,T\right), \label{i1} \\
		\nabla u \cdot \nu = \nabla \sinh(-\Delta u) \cdot \nu = 0 \;\; \mbox{ on } \; \Sigma_T \equiv \partial \Omega \times \left( 0, T \right), \label{i2} \\
		u(x,0) = u_0(x) \;\; \mbox{ on } \; \Omega, \label{i3}
	\end{align}
	where $\nu$ is the unit outward normal vector to the boundary.
	
	This equation comes from the study of crystal surfaces. If the space dimension is $N=1$, then equation \eqref{i1} arises as the limit of a microscopic Markov jump process, with a quadratic interaction potential and with Metropolis-type transition rates.\cite{GKLLM} In particular, let $h$ denote the height of the crystal, then the resulting equation for surface growth with Metropolis-type rates is given by,
	
	\begin{equation}
		\partial_t h = \frac{1}{2} \partial_x \left[ e^{-\partial^3_x u} - e^{\partial^3_x u} \right]. \label{imp}
	\end{equation}
	Upon setting $u = \partial_x h$, we arrive at our equation of interest,
	\begin{equation}
		\partial_t u = \partial_{xx} \left( \sinh(-\partial_{xx} u) \right) \label{met}
	\end{equation}

	The argument in \cite{GKLLM} relies on assuming that in local equilibrium, the Markov process's distribution is a local Gibbs measure. This assumption is correct for Arrhenius rates, which gives rise to a similar PDE,
	\begin{equation}
		\frac{\partial u}{\partial t} = \Delta e^{-\Delta u}. \label{arh}
	\end{equation}
	However, this assumption is not correct for the Metropolis rates, and only appears to be valid when the temperature goes to infinity. \cite{AK}  Without an explicit representation for the LE distribution, one cannot write down the PDE explicitly. More recently the author in \cite{AK} developed a method to determine the correct PDE numerically at any temperature. This was done by computing a temperature-dependent multiplicative correction to the current in the approximate PDE \eqref{met}, and it takes the form,
	\begin{equation}
		\partial_t h = - \partial_x\left(\sigma_K(h_{xxx})\sinh(Kh_{xxx})\right). 
	\end{equation}
	Here $\sinh(Kh_{xxx})$ is the macroscopic current under the local Gibbs assumption, $K$ is the inverse of the temperature, and $\sigma_K$ is the temperature dependent correction, which converges to $1$ as $K$ tends to zero, giving back equation \eqref{met}. \cite{AK}

	The Models obtained from the Metropolis rates unfortunately do not generalize well in higher dimensions, and the full multidimensional PDE has not yet been derived. This is quite different the analogous exponential PDE derived with Arrhenius-type rates, which does generalize to higher dimensions. \cite{MW}
	
	We also note that these PDE's are arising by applying a nonstandard scaling regime to the microscopic dynamics. Using a more standard scaling regime, the exponential will become linear. For the Arrhenius rates, this more standard approach has a limitation in that it treats local maxima and minima of $h$ symmetrically, whereas the exponential PDE derived using the nonstandard scaling does not. \cite{CLLMW} 
	
	For the Arrhenius rates equation, \eqref{arh} it was observed in \cite{LX} that one had to allow the possibility that the exponent could be a measure-valued function if the space dimension greater than 1. This is due to the lack of estimates for the exponent term. To remove the singularity in the exponent, one must impose a smallness condition on the initial data \cite{LS,PX}. For the exponential PDE \eqref{i1}, we are able to improve this, and find that the exponent is a $\wot$ function without any smallness assumptions on the given data, and, in particular, is not a measure.
	
	 What is interesting here is that this improvement still holds for nonlinear exponents, and is only due to the difference between the hyperbolic sine verses the standard exponential. In fact, even if we consider the p-laplacian as our exponent, (while the equation \eqref{met} came from an $L^2$ quadratic interaction, the p-laplacian comes from an $L^p$ interaction ), and have the equation, 
	\begin{equation}
		\partial_t h - \Delta \sinh( - \Delta_p h ) = 0 
	\end{equation}
	We can multiply through this equation by $- \Delta_p h $ and integrate to obtain,
	\begin{equation}
		\frac{1}{p} \frac{d}{dt} \int_{\Omega} \left| \nabla u \right|^p dx + \int_{\Omega} \cosh( - \Delta_p u ) \left| \nabla \Delta_p u \right|^2 dx = 0.
	\end{equation}
	Upon integrating with respect to time and using the fact that $\cosh(s) \geq 1$ we obtain,
	\begin{equation}
		\sup_{ 0 \leq t \leq T } \int_{\Omega} \left| \nabla u \right|^p dx + \int_{\ot} \left| \nabla \Delta_p u \right|^2 dx dt \leq \int_{\Omega} \left| \nabla u_0 \right|^p dx
	\end{equation}
	
	In this article we are interesting in solidifying the differences arising from the exponentials, and we will prove the unique existence of strong solutions to \eqref{i1}-\eqref{i3} in any dimension $N \geq 2$.

	
	Before we state our main theorem, we give the following definition of a strong solution,
	\begin{defn}\label{def1}
		We say a pair of functions $\left(u,w\right)$ is a strong solution of \ref{i1} - \ref{i3} if the following conditions hold: \\
		(D1) $u\in W^{1,2}(\Omega_T)\cap  L^2\left(0,T;W^{2,2}\left(\Omega\right)\right) $ with $ \partial_t u, \left| \nabla u \right| \in L^{\infty}\left(0,T;L^2(\Omega)\right)$, and $w \in W^{1,2}\left(\Omega_T\right)$ is such that $\sinh\left(w\right) \in L^2\left(0,T;W^{2,2}\left(\Omega\right)\right)$, \\
		(D2) we have,
		\begin{gather}
			\partial_t u - \Delta \sinh\left(w\right) = 0 \; \mbox{ a.e. on } \; \Omega_T, \label{defeq1} \\
			-\Delta u = w \; \mbox{ a.e. on } \; \Omega_T, \label{defeq2} \\
			\nabla u \cdot \nu = \nabla \sinh\left( w\right)  \cdot \nu = 0 \; \mbox{ a.e. on } \; \Sigma_T, \label{defeq3}\\
			u(x,0)=u_0(x), \label{defeq4}
		\end{gather}
		where the initial condition \eqref{defeq4} is satisfied in the space $C\left(\left[0,T\right];L^2\left(\Omega\right)\right)$.
	\end{defn}
	
	Our result is the following,
	
	\begin{thm}[Main Theorem]\label{thm1}
		Assume: \\
		(H1) $\Omega$ is a bounded domain in $\mathbb{R}^N$ which is either convex or has $C^2$ boundary; \\
		(H2) $N \geq 2$; \\
		(H3) $u_0 \in W^{2,2}\left(\Omega\right)$ is such that $\cosh\left(-\Delta u_0 \right) \in L^1(\Omega),\ \sinh\left(-\Delta u_0 \right) \in W^{2,2}\left(\Omega\right)$, and $\nabla u_0\cdot\nu = 0$, $\nabla\sinh(-\Delta u_0)\cdot\nu=0$ on $\partial\Omega$. \\
		Then there is a unique global strong solution to equation \ref{i1} - \ref{i3} in the sense of Definition \ref{def1}.
	\end{thm}
	
	By a global strong solution, we mean that for each $T>0$ there is a strong solution $u$ to (1.1)-(1.3) on $\Omega_T$. 
	
	Our result here indicates that the two principal nonlinear terms in equation (1.1) have a balancing effect. Indeed, according to  \cite{LX},
	the term $e^{-\Delta u}$ can behave well even if $-\Delta u$ is a Radon measure with the support of its singular part contained in the set where the absolutely continuous part is negative infinity. This suggests that  good estimates on $e^{-\Delta u}$ do not control what may happen on the set $\{\Delta u=\infty\}$. Naturally, one would expect that estimates for the second nonlinear term $e^{\Delta u}$ in our equation could make up for this.
	Our analysis here shows that this is exactly what has happened.   Note that measure exponents do not arise in the one-dimensional case. See \cite{GLL2,GLLM}.

	
	The uniqueness assertion in Theorem \ref{thm1} is simple. We will prove it now. Let $v$ be a second solution of \ref{defeq1} - \ref{defeq4}. Then we have,
	\begin{equation}
		\frac{\partial}{\partial t} \left( u - v \right) = \Delta\sinh(-\Delta u) - \Delta\sinh(-\Delta v).
	\end{equation}
	Multiply through this equation by $u-v$ and then integrate the resulting equation over $\Omega$ to get, 
	\begin{equation}
		\begin{aligned}
			\frac{1}{2} \frac{d}{dt} \int_{\Omega} \left( u - v\right)^2 dx &= \int_{\Omega} \left( \Delta\sinh(-\Delta u) - \Delta\sinh(-\Delta v) \right) \left( u - v \right) dx \\
			& = - \int_{\Omega}\left( \sinh(-\Delta u) - \sinh(-\Delta v) \right) \left( -\Delta u - (-\Delta v) \right) dx \\
			& \leq 0.
		\end{aligned}
	\end{equation}
	Here we have used the fact that $\sinh$ is an increasing function. Integrate the above inequality with respect to $t$ to complete the proof.
	
	\subsection{A-Priori Estimates for smooth solutions}
	
	The core of our approach to this problem lies in the following a-priori estimates, which resemble those in \cite{LX}. Towards our first estimate we square both sides of \ref{defeq1} and integrate over $\Omega$ to obtain,
	\begin{equation}
		\int_{\Omega} \left( \partial_t u \right)^2 dx - 2 \int_{\Omega} \partial_t u \Delta \sinh\left(w\right) dx + \int_{\Omega} \left( \Delta \sinh\left(w\right)\right)^2 dx = 0.
	\end{equation}
	For the second integral in the above equation we have,
	\begin{align}
		-2\int_{\Omega} \partial_t u \Delta \sinh \left(w\right) dx & = -2 \int_{\Omega} \partial_t \Delta u \sinh\left(w\right) dx \nonumber\\
		&= 2 \frac{d}{dt} \int_{\Omega} \cosh\left(w\right) dx.
	\end{align}
	Putting this back into the previous equation we get,
	\begin{equation}
		\int_{\Omega} \left( \partial_t u \right)^2 dx + 2 \frac{d}{dt} \int_{\Omega} \cosh \left( w \right) dx + \int_{\Omega} \left( \Delta \sinh \left( w \right) \right)^2 dx = 0.
	\end{equation}
	Upon integrating with respect to time, we obtain our first estimate,
	\begin{equation}
		\begin{split}
			\int_{\Omega_T}  \left( \partial_t u \right)^2 dxdt + 2 \sup_{ 0 \leq t \leq T } \int_{\Omega}  \cosh \left( w \right) dx \\
			+ \int_{\Omega_T} \left( \Delta \sinh \left( w \right) \right)^2 dx dt 
			\leq 2 \int_{\Omega} \cosh \left( - \Delta u_0 \right) dx \label{apest1}
		\end{split}
	\end{equation}
	Next, we take the gradient of both sides of equation \ref{defeq1}, and then we take the dot product of the resulting equations with $\nabla u$ and finally integrate over $\Omega$ to get,
	\begin{equation}
		\frac{1}{2} \frac{d}{dt} \int_{\Omega} \left\lvert \nabla u \right\rvert^2 dx - \int_{\Omega} \nabla \left( \Delta \sinh \left( w \right) \right) \cdot \nabla u dx = 0. \label{hlp1}
	\end{equation}
	For the second integral in this equation we have,
	\begin{align*}
		- \int_{ \Omega } \nabla \left( \Delta \sinh \left( w \right) \right) \cdot \nabla u dx & = - \int_{ \Omega } \nabla \sinh \left( w \right) \cdot \nabla \Delta u dx \\
		& = \int_{ \Omega } \cosh \left( w \right) \left\lvert \nabla w \right\rvert^2 dx \\
		& \geq \int_{ \Omega } \left\lvert \nabla w \right\rvert^2 dx.
	\end{align*}
	Putting this back into \ref{hlp1}, and then integrating with respect to $t$ gives us our second estimate,
	\begin{equation}
		\frac{1}{2} \sup_{ 0 \leq t \leq T } \int_{\Omega} \left\lvert \nabla u \right\rvert^2 dx + \int_{\Omega_T} \left\lvert \nabla w \right\rvert^2 dxdt \leq \int_{\Omega} \left\lvert \nabla u_0 \right\rvert^2 dx. \label{apest2}
	\end{equation}
	Towards our third estimate we first take the derivative of \ref{defeq1} with respect to $t$. Then we multiply through the resulting equation by $\partial_t u$ and integrate over $\Omega$ to obtain,
	\begin{equation}
		\frac{1}{2} \frac{d}{dt} \int_{\Omega} \left( \partial_t u \right)^2 dx - \int_{ \Omega } \partial_t \Delta \sinh \left( w \right) \partial_t u dx = 0.
	\end{equation}
	For the second integral in this equation we calculate,
	\begin{align*}
		- \int_{ \Omega } \partial_t \Delta \sinh \left( w \right) \partial_t u dx & = - \int_{ \Omega } \partial_t \sinh \left( w \right) \partial_t \Delta u dx \\
		& = \int_{ \Omega } \partial_t \sinh \left( w \right) \partial_t w dx \\
		& = \int_{ \Omega } \cosh \left( w \right) \left( \partial_t w \right)^2 dx \\
		& \geq \int_{ \Omega } \left( \partial_t w \right)^2 dx.
	\end{align*}
	Putting this into the original equation and then integrating with respect to $t$ gives our third estimate,
	\begin{equation}
		\frac{1}{2} \sup_{ 0 \leq t \leq T } \int_{ \Omega } \left( \partial_t u \right)^2 dx + \int_{ \Omega_T } \left( \partial_t w \right)^2 dxdt \leq \int_{\Omega} \left( \Delta \sinh \left( - \Delta u_0 \right) \right)^2 dx. \label{apest3}
	\end{equation}
	
	Finally, we integrate \eqref{i1} with respect to $x$ over $\Omega$ to get
	$$\frac{d}{dt}\int_{ \Omega }u(x,t)dx=0,$$
	from whence follows
	\begin{equation}
		\int_{ \Omega }u(x,t)dx=\int_{ \Omega }u_0(x)dx.
	\end{equation}
	Similarly, we can deduce from \eqref{defeq2} that
	\begin{equation}
		\int_{ \Omega }w(x,t)dx=0.
	\end{equation}
	
	We shall see that the preceding a-priori estimates combined with relevant interpolation inequalities for Sobolev spaces and Lemma \ref{lwt} below imply (D1) in the definition \ref{def1}.

	A solution to \ref{i1} - \ref{i3} will be constructed as the limit of a sequence of approximate solutions. In section 2 we will present our approximate problems, and we establish the existence of a classical solution for these problems. We then form a sequence of approximate solutions based upon implicit discretization in the time variable. Section 3 is then devoted to the proof of the discrete versions of the estimates obtained in Subsection 1.2. These estimates are then shown to be enough to justify passing to the limit.
	
	\vskip2mm
	
	\section{Approximate Problems}
	In this section we will prove the existence of solutions to our approximate problems. Before we present our approximate problems, we will state a few preparatory lemmas, which will be used throughout the article. 
	
	We begin by detailing the relevant spaces of functions which will be used throughout. For $p\geq1$, we will let $L^p(\Omega)$ denote the Banach space of measurable functions on $\Omega$ that are p-integrable. We then denote by $W^{k.p}(\Omega)$ the Banach space of functions $f \in L^p(\Omega)$ so that the weak derivatives $D^{\alpha} f \in L^p(\Omega)$ for all multi-indices $\alpha$ with $\left| \alpha \right| \leq k$. We will also denote the space of H\"older continuous functions with exponent $\beta$ via $C^{\beta}\left(\overline{ \Omega} \right)$ for some $\beta \in (0,1)$.

	\begin{lemma}\label{lwt}
		Let $\Omega$ be a bounded domain in $\mathbb{R}^N$. If $ \Omega $ is convex, then,
		\begin{equation}
			\int_{\Omega} \left( \Delta u \right)^2 dx \geq \int_{\Omega} \left\lvert \nabla^2 u \right\rvert^2 dx
		\end{equation}
		for all $ u \in W^{2,2} \left( \Omega \right) $ with $ \nabla u \cdot \nu = 0 $ on $ \partial \Omega $.
		If $\po$ is $C^2$, then there is a positive constant $c$
		depending only on $N, \Omega$, and the smoothness of the boundary
		such that
		\begin{equation}
			\int_\Omega(\Delta u)^2dx+\int_\Omega|\nabla
			u|^2dx\geq
			c\int_\Omega|\nabla^2u|^2
			dx \ \
		\end{equation}
		for  all $u\in W^{2,2}(\Omega)$ with $\nabla u\cdot\nu=0$ on
		$\po$.
	\end{lemma}
	
	For more information on this lemma we refer the reader to \cite{X.Xu}.
	
	Next, we present some relevant interpolation inequalities for Sobolev spaces,
	\begin{lemma}\label{l22}
		Let $ \Omega $ be a bounded domain in $\mathbb{R}^N$. Then,\\
		(1) $\left\lVert f \right\rVert_{q,\Omega} \leq \epsilon \left\lVert f \right\rVert_{r,\Omega} + \epsilon^{\sigma} \left\lVert f \right\rVert_{p,\Omega}$, where $\epsilon > 0, \; p \leq q < r$, and $\sigma = \left( \frac{1}{p} - \frac{1}{q} \right)/\left(\frac{1}{q} - \frac{1}{r}\right)$;\\
		(2) If $\partial\Omega$ is $C^2$, for each $ \epsilon > 0 $ and each $ p \in [2, 2^{*} )$, where $2^{*} = \frac{2N}{N-2}$ if $ N > 2 $ and any number bigger than $2$ if $ N = 2 $, there is a positive number $ c = c(\epsilon, p)$ such that,
		\begin{equation}
			\begin{aligned}
				\left\lVert f \right\rVert_p &\leq \epsilon \left\lVert \nabla f \right\rVert_2 + c \left\lVert f \right\rVert_1, \\
				\left\lVert \nabla g \right\rVert_p &\leq \epsilon \left\lVert \nabla^2 g \right\rVert_2 + c \left\lVert g \right\rVert_1,
			\end{aligned}
		\end{equation}
		for all $ f \in \wot $ and $ g \in  \wtt $.
	\end{lemma} 
	
	Our existence assertion is based on the following fixed point theorem, which is often called the Leray-Schauder Theorem. (\cite{GT}, p.280) 
	\begin{lemma}
		Let $B$ be a map from a Banach space $\mathbb{X}$ into itself. Assume:\\
		(1) $B$ is continuous,\\
		(2) the images of bounded sets under $B$ are precompact,\\
		(3) there exists a constant $c$ so that,
		\begin{equation}
			\left\lVert z \right\rVert_{ \mathbb{X} } \leq c
		\end{equation}
		for all $ z \in \mathbb{X}$, and $ \sigma \in \left( 0, 1 \right) $ satisfying,
		\begin{equation}
			z = \sigma B(z).
		\end{equation}
		Then $B$ has a fixed point.
	\end{lemma}

	Next we collect a few useful elementary inequalities.
	
	\begin{lemma}
		(1) If $f$ is an increasing function and $F$ an anti-derivative of $f$, then,
		\begin{equation}
			f\left(s\right) \left( s-t \right) \geq F(s) - F(t).
		\end{equation}
		(2) for $a,b\in[0, \infty)$ there hold
		\begin{equation}
			\begin{aligned}
				\left( a + b \right)^{\alpha} &\leq a^{\alpha} + b^{\alpha} \; \mbox{ if } 0 < \alpha \leq 1, \\
				\left( a + b \right)^{\alpha} &\leq 2^{\alpha - 1} \left(  a^{\alpha} + b^{\alpha}\right) \; \mbox{ if } \; \alpha > 1.
			\end{aligned}
		\end{equation}
	\end{lemma}
	
	Finally we will also make use of the following well known estimate, 
	\begin{lemma}\label{l26}
		Let $w \in\wot$ be a weak solution of the boundary value problem,
		\begin{equation}
			\begin{aligned}
				- \Delta w & = f \;\; \mbox{ in } \; \Omega, \\
				\nabla w \cdot \nu & = 0 \;\; \mbox{ on } \; \partial \Omega.
			\end{aligned}
		\end{equation}
		Then for each $p > \frac{N}{2}$ there is a positive number $c=c\left(N,p,\Omega\right)$ so that,
		\begin{equation}
			\left\lVert w \right\rVert_{\infty,\Omega} \leq c \left\lVert w \right\rVert_{1,\Omega} + c \left\lVert f \right\rVert_{p,\Omega}.
		\end{equation}
	\end{lemma}

	We now introduce our approximate problems. The approximation scheme is based on an implicit time discretization. We let $ \tau > 0 $, and $ v \in L^{ \infty } \left( \Omega \right) $ be given,and we consider the following boundary-value problem,
	
	\begin{gather}
		\frac{ u - v }{\tau} - \Delta \sinh \left( w \right) + \tau w = 0 \;\; \mbox{ on } \; \Omega, \label{app1} \\
		- \Delta u + \tau u = w \;\; \mbox{ on } \; \Omega, \label{app2} \\
		\nabla w \cdot \nu = \nabla u \cdot \nu = 0 \;\; \mbox{ on } \; \partial \Omega. \label{app3} 
	\end{gather}
	
	\begin{prop}\label{p21}
		There exists a weak solution $(u, w)$ to equation \eqref{app1}-\eqref{app3} in the space $\left(\wtt\cap C^\alpha(\overline{\Omega})\right)^2$ for some $\alpha\in (0,1)$.
	\end{prop}
	
	\begin{proof}
		To prove the existence of a solution, we will employ the Leray - Schauder Theorem. To this end we define a mapping from $L^{\infty} \left( \Omega \right) $ into itself as follows: for $ \varphi \in L^{ \infty } \left( \Omega \right) $ we first define $ u $ to be the weak solution to the problem,
		\begin{gather}
			- \Delta u + \tau u = \varphi \;\; \mbox{ in } \; \Omega, \label{LS1} \\
			\nabla u \cdot \nu = 0 \;\; \mbox{ on } \; \partial \Omega. \label{LS2}
		\end{gather}
		From the classical theory for linear elliptic equation there is a unique weak solution $u$ in the space $\wot$. Furthermore,  
		$u$ is H\"{o}lder continuous in $\overline{\Omega}$.
		Using this $u$ we then form the problem,
		\begin{gather}
			- \mbox{div} \left( \cosh \left( \varphi \right) \nabla w \right) + \tau w = - \frac{ u - v }{ \tau } \;\; \mbox{ on } \; \Omega, \label{LS3} \\
			\nabla w \cdot \nu = 0 \;\; \mbox{ on } \; \partial \Omega. \label{LS4}
		\end{gather}
		Since $ \varphi \in L^{ \infty } \left( \Omega \right) $ and $\cosh\varphi \geq 1$, equation $ \eqref{LS3} $ is uniformly elliptic, and so from the classical existence theory there is a unique weak solution $w \in\wot\cap C^{\beta} \left( \overline{\Omega} \right) $ for some $ \beta \in \left( 0, 1 \right) $. We set $B \left( \varphi \right) = w $. Clearly, we can conclude that $B$ is well defined, continuous, and maps bounded sets into precompact ones. We still need to show that there is a positive number c so that 
		\begin{equation}
			\left\lVert w \right\rVert_{ \infty, \Omega } \leq c,
		\end{equation}
		for all $ w \in L^{ \infty } \left( \Omega \right) $ and $ \sigma \in \left( 0, 1 \right) $ satisfying,
		\begin{equation}
			w = \sigma B \left( w \right).
		\end{equation}
		This is equivalent to the boundary value problem,
		\begin{gather}
			- \mbox{div} \left( \cosh \left( w \right) \nabla w \right) + \tau w = - \sigma \frac{ u - v }{ \tau } \;\; \mbox{ in } \; \Omega, \label{sgm1}\\
			- \Delta u + \tau u = w \;\; \mbox{ in } \; \Omega, \label{sgm2}\\
			\nabla u \cdot \nu = \nabla w \cdot \nu = 0 \;\; \mbox{ on } \; \partial \Omega. \label{sgm3}
		\end{gather}
		First use $w$ as a test function in \eqref{sgm1} and use the fact that $\cosh(w)\geq 1$ to get,
		\begin{equation}
			\begin{aligned}
				\int_{\Omega} |\nabla w |^2 dx + \tau\int_{\Omega} w^2 dx &\leq - \frac{1}{\tau} \int_{\Omega} \left( u - v \right) wdx \\
				& = - \frac{1}{\tau} \int_{\Omega} uw dx + \frac{1}{\tau} \int_{\Omega} v w dx.\label{sgm4}
			\end{aligned}
		\end{equation}
		We now use $u$ as a test function in \eqref{sgm2}, yielding,
		\begin{equation}\label{sgm5}
			\int_{\Omega} wu dx = \int_{\Omega} |\nabla u |^2 dx + \tau \int_{\Omega} u^2 dx \geq 0.
		\end{equation}
		Using the above equation in \eqref{sgm4}, we are able to derive,
		\begin{equation}
			\int_{\Omega} |\nabla w |^2 dx + \tau\int_{\Omega} w^2 dx \leq \frac{1}{\tau} \int_{\Omega} v w dx.
		\end{equation}
		Using the above equation and \eqref{sgm5} we then have,
		\begin{equation}
			\begin{aligned}
				\int_{\Omega} |\nabla w |^2 dx + \int_{\Omega} w^2 dx &\leq c(\tau) \int_{\Omega} v^2 dx \\
				\int_{\Omega} |\nabla u |^2 dx + \int_{\Omega} u^2 dx &\leq c(\tau) \int_{\Omega} w^2 dx.\label{sgm6}
			\end{aligned}
		\end{equation}
		Then for each $p>2$, we use the function $\left\lvert w \right\rvert^{p-2} w$ as a test function in \ref{sgm1} to obtain,
		\begin{equation}
			\begin{aligned}
				\left(p-1\right)\int_{\Omega} \left\lvert w \right\rvert^{p-2} \cosh\left(w\right) \left\lvert \nabla w \right\rvert^2 dx +  \tau \int_{\Omega} \left\lvert w \right\rvert^p dx &\leq \int_{\Omega} \left\lvert \frac{u-v}{\tau} \right\rvert \left\lvert w \right\rvert^{p-1} dx \\
				& \leq \left\lVert \frac{u-v}{\tau} \right\rVert_{p,\Omega}\left\lVert w \right\rVert^{p-1}_{p,\Omega}.
			\end{aligned}
		\end{equation}
		After dropping the first integral, we obtain,
		\begin{equation}
			\tau \left\lVert w \right\rVert_{p,\Omega} \leq \left\lVert \frac{u-v}{\tau} \right\rVert_{p,\Omega}.
		\end{equation}
		Letting $p \rightarrow \infty$ then gives,
		\begin{equation}
			\tau \left\lVert w \right\rVert_{\infty,\Omega} \leq \left\lVert \frac{u-v}{\tau} \right\rVert_{\infty,\Omega}. \label{hp}
		\end{equation}
		From Lemma \ref{l26} we then get that for each $q > \max\left\{\frac{N}{2},2\right\}$ there is a positive number $c=c\left(N,\Omega,\tau\right)$ so that,
		\begin{equation}
			\left\lVert u \right\rVert_{\infty,\Omega} \leq c \left\lVert u \right\rVert_{1,\Omega} + c \left\lVert w \right\rVert_{q,\Omega} \leq c \left\lVert w \right\rVert_{q,\Omega}.
		\end{equation}
		The last step is due to \eqref{sgm6}.
		Using this in conjunction with \ref{hp} we have,
		\begin{equation}
			\begin{aligned}
				\left\lVert w \right\rVert_{\infty,\Omega} & \leq c \left\lVert u \right\rVert_{\infty,\Omega} + c \left\lVert v \right\rVert_{\infty,\Omega} \\
				& \leq c \left\lVert w \right\rVert_{q,\Omega} + c \left\lVert v \right\rVert_{\infty,\Omega} \\
				& \leq \epsilon \left\lVert w \right\rVert_{\infty,\Omega} + c\left(\epsilon\right) \left\lVert w \right\rVert_{1,\Omega} + c \left\lVert v \right\rVert_{\infty,\Omega} 
			\end{aligned}
		\end{equation}
		Taking $\epsilon$ suitably small we finally get,
		\begin{equation}
			\left\lVert w \right\rVert_{\infty,\Omega} \leq c \left\lVert v \right\rVert_{\infty,\Omega} +c\left\lVert w \right\rVert_{1,\Omega}\leq c.
		\end{equation}
		Here we have used \eqref{sgm6}.
		
		The fact that both $u$ and $\sinh w$ lie in $\wtt$ is a consequence of Lemma \ref{lwt}. Then we can deduce $w\in\wtt$ from the boundedness of $w$. The proof is complete.
	\end{proof}
	
	\section{Proof of the Main Theorem}
	
	The proof of our main Theorem will be accomplished in several stages. First we present the time discretized problem. The existence of a solution to this problem is dependent on our approximate problem from section 3. We then derive estimates similar to our apriori estimates, and show that this is enough to justify in passing to the limit. 
	
	Let $ T > 0 $ be given. For each $ j \in \{ 1, 2, 3,... \} $ we divide the time interval $ \left[ 0, T \right] $ into $ j $ equal sub-intervals. Set
	\begin{equation}
		\tau = \frac{T}{j},\ \  t_k = k \tau,\  k=0, 1,\cdots, j.
	\end{equation}
	Let $ u_0 $ be given, satisfying (H3). It is not difficult to see from (H3) that $u_0\in L^{\infty}(\Omega)$. Thus by Proposition \ref{p21} we can recursively solve the system,
	\begin{gather}
		\frac{ u_k - u_{k-1}}{\tau} - \Delta \sinh \left( w_k \right) + \tau w_k = 0 \;\; \mbox{ in } \; \Omega, \label{td1} \\
		- \Delta u_k + \tau u_k = w_k \;\; \mbox{ in } \; \Omega, \label{td2} \\
		\nabla u_k \cdot \nu = \nabla w_k \cdot \nu = 0 \;\; \mbox{ on } \; \partial \Omega. \label{td3} 
	\end{gather}
	We can pick $w_0$ and $u_{-1}$ so that the equations
	\begin{equation}
		\begin{aligned}
			\frac{u_0 - u_{-1}}{\tau} -& \Delta \sinh\left(w_0\right) + \tau w_0 = 0 \;\; \mbox{ in } \; \Omega, \label{wu1}\\
			- &\Delta u_0 + \tau u_0 = w_0 \;\; \mbox{ in } \; \Omega 
		\end{aligned}
	\end{equation}
	are satisfied. This together with (H3) implies that \eqref{td1}-\eqref{td3} still hold for $k=0$.
	
	Next, we define the functions $	\tilde{u}_j \left( x, t \right),\bar{u}_j \left( x, t \right), \tilde{w}_j \left( x, t \right), 	\bar{w}_j \left( x, t \right)$ on $ \Omega_T$ as follows: For each $(x,t)\in \Omega_T$
	there is $k$ such that $t \in (t_{k-1}, t_k] $. Subsequently, set
	\begin{gather}
		\tilde{u}_j \left( x, t \right) = \frac{ t - t_{k-1} }{\tau} u_k \left( x \right) + \left( 1 - \frac{ t - t_{k-1} }{\tau} \right) u_{k-1} \left( x \right),  \label{usquig} \\
		\bar{u}_j \left( x, t \right) = u_k \left( x \right), \label{ubar} \\
		\tilde{w}_j \left( x, t \right) = \frac{ t - t_{k-1}}{\tau} w_k \left( x \right) + \left( 1 - \frac{ t - t_{k-1}}{\tau} \right) w_{k-1} \left( x \right) \label{wsquig} \\
		\bar{w}_j \left( x, t \right) = w_k \left( x \right). \label{wbar} 
	\end{gather}
	Using these functions we may then write our discretized system as,
	\begin{gather}
		\partial_t \tilde{u}_j - \Delta \sinh \left( \bar{w}_j \right) + \tau \bar{w}_j = 0 \;\; \mbox{ on } \; \Omega_T, \label{ttd1} \\
		- \Delta \bar{u}_j + \tau \bar{u}_j = \bar{w}_j \;\; \mbox{ on } \; \Omega_T. \label{ttd2} 
	\end{gather}
	We now proceed to derive the discrete analogues of our A-priori estimates. 
	First we have the discrete version of \ref{apest1}. 
	\begin{prop}\label{p31}
		For the functions $\tilde{u}_j$, $\bar{u}_j$, $\tilde{w}_j$, and $\bar{w}_j$ given by \eqref{usquig}-\eqref{wbar}, we have the following estimate:
		\begin{equation}
			\begin{aligned}
				\int_{ \Omega_T } & \left( \frac{ \partial \tilde{u}_j }{ \partial t } \right)^2 dx dt  + \int_{ \Omega_T } \left( \Delta \sinh \left( \bar{w}_j \right) \right)^2 dx dt + \tau^2 \int_{ \Omega_T } \bar{w}^2_j dx dt \\
				& + 2 \tau \int_{ \Omega_T } \left\lvert \nabla \sinh \left( \bar{w}_j \right) \right\rvert^2 dx dt + 2 \tau^2 \int_{ \Omega_T } \bar{w}_j \sinh \left( \bar{w}_j \right) dx dt \\
				& + 2 \max_{0 \leq t \leq T } \int_{ \Omega } \cosh \left( \bar{w}_j \right) dx + 2 \max_{0 \leq t \leq T } \left(\tau \int_{ \Omega } \left\lvert \nabla \bar{u}_j \right\rvert^2 dx + \tau^2 \int_{ \Omega } \bar{u}^2_j dx \right) \\
				& + 2 \tau \int_{ \Omega_T } \left\lvert \nabla \bar{w}_j \right\rvert^2 dxdt \\
				& \leq 2 \int_{ \Omega } \cosh \left( w_0 \right) dx + 2 \tau\int_{ \Omega } \left\lvert \nabla u_0 \right\rvert^2 dx + 2 \tau^2 \int_{ \Omega } u^2_0 dx \label{tdp1.1}
			\end{aligned}
		\end{equation}
	\end{prop}
	\begin{proof}
		First we square both sides of equation $ \eqref{td1} $, and the integrate the resulting equation over $ \Omega $ to obtain,
		\begin{equation}
			\begin{aligned}
				\int_{ \Omega } & \left( \frac{ u_k - u_{k-1} }{\tau} \right)^2 dx + \int_{ \Omega } \left( \Delta \sinh \left( w_k \right) \right)^2 dx + \tau^2 \int_{ \Omega } w^2_k dx \\
				& - \frac{ 2 }{ \tau } \int_{ \Omega } \left( u_k - u_{k-1} \right) \Delta \sinh \left( w_k \right) dx + 2 \int_{ \Omega } \left( u_k - u_{k-1} \right) w_k dx \\
				& - 2 \tau \int_{ \Omega } w_k \Delta \sinh \left( w_k \right) dx = 0. \label{tdp1.2}
			\end{aligned}
		\end{equation}
		The first three integrals in this equation are good, we need only worry about the final three. For the fourth integral  in the above equation we calculate,
		\begin{equation}
			\begin{aligned}
				- \frac{ 2 }{ \tau } \int_{ \Omega } & \left( u_k - u_{k-1} \right) \Delta \sinh \left( w_k \right) dx = - \frac{2}{\tau} \int_{ \Omega } \Delta \left( u_k - u_{k-1} \right) \sinh \left( w_k \right) dx \\
				& = \frac{ 2 }{ \tau } \int_{ \Omega } \left[ - \tau \left( u_k - u_{k-1} \right) + \left( w_k - w_{k-1} \right) \right] \sinh \left( w_k \right) dx \\
				& = - 2 \tau \int_{ \Omega } \left( \frac{u_k - u_{k-1}}{ \tau } \right) \sinh \left( w_k \right) dx + \frac{2}{\tau} \int_{ \Omega } \left( w_k - w_{k-1} \right) \sinh \left( w_k \right) dx \\
				& = - 2 \tau \int_{ \Omega } \Delta \sinh \left( w_k \right) \sinh \left( w_k \right) dx + 2 \tau^2 \int_{ \Omega } w_k \sinh \left( w_k \right) dx \\
				&+ \frac{2}{ \tau } \int_{ \Omega } \left( w_k - w_{k-1} \right) \sinh \left( w_k \right) dx \\
				& = 2 \tau \int_{ \Omega } \left\lvert \nabla \sinh \left( w_k \right) \right\rvert^2 dx + 2 \tau^2 \int_{ \Omega } w_k \sinh \left( w_k \right) dx  + \frac{2}{\tau} \int_{ \Omega } \left( w_k - w_{k-1} \right) \sinh \left( w_k \right) dx \\
				& \geq 2 \tau \int_{ \Omega } \left\lvert \nabla \sinh \left( w_k \right) \right\rvert^2 dx + 2 \tau^2 \int_{ \Omega } w_k \sinh \left( w_k \right) dx \\
				&+ \frac{2}{\tau} \int_{ \Omega } \left( \cosh \left( w_k \right) - \cosh \left( w_{k-1} \right) \right) dx \label{tdp1.3}
			\end{aligned}
		\end{equation}
		Then for the fifth integral in equation $ \eqref{tdp1.2} $ we calculate,
		\begin{equation}
			\begin{aligned}
				2 \int_{ \Omega }  \left( u_k - u_{k-1} \right) w_k dx & = 2 \int_{ \Omega } \left( u_k - u_{k-1} \right) \left( - \Delta u_k + \tau u_k \right) dx \\
				& = - 2 \int_{ \Omega } \left( u_k - u_{k-1} \right) \Delta u_k dx + 2 \tau \int_{ \Omega } \left( u_k - u_{k-1} \right) u_k dx \\
				& = 2 \int_{ \Omega } \left( \nabla u_k - \nabla u_{k-1} \right) \nabla u_k dx + 2 \tau \int_{ \Omega } \left( u_k - u_{k-1} \right) u_k dx \\
				& \geq 2 \int_{ \Omega } \left( \left\lvert \nabla u_k \right\rvert^2 - \left\lvert \nabla u_{k-1} \right\rvert^2 \right) dx + 2 \tau \int_{ \Omega } \left( u^2_k - u^2_{k-1} \right) dx. \label{tdp1.4}
			\end{aligned}
		\end{equation}
		Finally for the last integral in equation $ \eqref{tdp1.2} $ we calculate,
		\begin{equation}
			\begin{aligned}
				- 2 \tau \int_{ \Omega } w_k \Delta \sinh \left( w_k \right) dx & = 2 \tau \int_{ \Omega } \nabla w_k \cdot \nabla \sinh \left( w_k \right) dx \\
				& = 2 \tau \int_{ \Omega } \nabla w_k \cdot \left( \cosh \left( w_k \right) \nabla w_k \right) dx \\
				& = 2 \tau \int_{ \Omega } \cosh \left( w_k \right) \left\lvert \nabla w_k \right\rvert^2 dx \\
				& \geq 2 \tau \int_{ \Omega } \left\lvert \nabla w_k \right\rvert^2 dx. \label{tdp1.5}
			\end{aligned}
		\end{equation}
		We plug each of the integrals $ \eqref{tdp1.3} $ - $ \eqref{tdp1.5} $ back into equation $ \eqref{tdp1.2} $ to then get the inequality,
		\begin{equation}
			\begin{aligned}
				\int_{ \Omega } & \left( \frac{ u_k - u_{k-1} }{ \tau } \right)^2 dx + \int_{ \Omega } \left( \Delta \sinh \left( w_k \right) \right)^2 dx + \tau^2 \int_{ \Omega } w^2_k dx + 2 \tau \int_{ \Omega } \left\lvert \nabla \sinh \left( w_k \right) \right\rvert^2 dx \\
				& + 2 \tau^2 \int_{ \Omega } w_k \sinh \left( w_k \right) dx + \frac{ 2 }{ \tau } \int_{ \Omega } \left( \cosh \left( w_k \right) - \cosh \left( w_{k-1} \right) \right) dx + 2 \tau \int_{ \Omega } \left\lvert \nabla w_k \right\rvert^2 dx \\
				& + 2 \int_{ \Omega } \left( \left\lvert \nabla u_k \right\rvert^2 - \left\lvert \nabla u_{ k-1} \right\rvert^2 \right) dx + 2 \tau \int_{ \Omega } \left( u^2_k - u^2_{k-1} \right) dx = 0. \label{tdp1.6}
			\end{aligned}
		\end{equation}
		Finally we multiply through this inequality by $ \tau $ and sum up over $ k $ to get the result.
	\end{proof}
	
	Next, we have the discrete version of \ref{apest2}.
	\begin{prop}\label{p32}
		Let the functions $\tilde{u}_j$, $\bar{u}_j$, $\tilde{w}_j$, and $\bar{w}_j$ be given by \eqref{usquig}-\eqref{wbar}, then we have the estimate:
		\begin{equation}
			\begin{aligned}
				\max_{0 \leq t \leq T } & \left( \int_{ \Omega } \left\lvert \nabla \bar{u}_j \right\rvert^2 dx + \tau \int_{ \Omega } \bar{u}^2_j dx \right) +\tau^3 \int_{ \Omega_T } \bar{u}^2_j dx dt \\
				&+ \int_{ \Omega_T } \left\lvert \nabla \bar{w}_j \right\rvert^2 dx dt + \tau \int_{ \Omega_T } \left( \Delta \bar{u}_j \right)^2 dx dt + 2\tau^2 \int_{ \Omega_T} \left\lvert \nabla \bar{u}_j \right\rvert^2 dxdt \\
				& \leq \int_{ \Omega } \left\lvert \nabla u_0 \right\rvert^2 dx + \tau \int_{ \Omega } u^2_0 dx. \label{tdp2.1}
			\end{aligned}
		\end{equation}
	\end{prop}
	\begin{proof}
		We take the gradient of both side of equation $ \eqref{td1} $, and then dot the resulting equation with $ \nabla u_k $ and integrate over $ \Omega $ to get,
		\begin{equation}
			\frac{1}{\tau} \int_{ \Omega } \nabla \left( u_k - u_{k-1} \right) \cdot \nabla u_k dx - \int_{ \Omega } \nabla \left( \Delta \sinh \left( w_k \right) \right) \cdot \nabla u_k dx + \tau \int_{ \Omega } \nabla w_k \cdot \nabla u_k dx = 0 \label{tdp2.2}
		\end{equation}
		For the first integral in the above equation we calculate,
		\begin{equation}
			\frac{1}{\tau} \int_{ \Omega } \nabla \left( u_k - u_{k-1} \right) \cdot \nabla u_k dx \geq \frac{1}{\tau} \int_{ \Omega } \left( \left\lvert \nabla u_k \right\rvert^2 - \left\lvert \nabla u_{k-1} \right\rvert^2 \right) dx. \label{tdp2.3}
		\end{equation}
		For the second integral in equation $ \eqref{tdp2.2} $ we calculate,
		\begin{equation}
			\begin{aligned}
				- \int_{ \Omega } & \nabla \left( \Delta \sinh \left( w_k \right) \right) \cdot \nabla u_k dx = - \int_{ \Omega } \nabla \sinh \left( w_k \right) \cdot \nabla \Delta u_k dx \\
				& = \int_{ \Omega } \nabla \sinh \left( w_k \right) \cdot \nabla \left( - \tau u_k + w_k \right) dx \\
				& = - \tau \int_{ \Omega } \nabla \sinh \left( w_k \right) \cdot \nabla u_k dx + \int_{ \Omega } \nabla \sinh \left( w_k \right) \cdot \nabla w_k dx \\
				& = \tau \int_{ \Omega } \Delta \sinh \left( w_k \right) u_k dx + \int_{ \Omega } \cosh \left( w_k \right) \left\lvert \nabla w_k \right\rvert^2 dx \\
				& = \tau \int_{ \Omega } \left[ \frac{ u_k - u_{k-1} }{ \tau } + \tau w_k \right] u_k dx + \int_{ \Omega } \cosh \left( w_k \right) \left\lvert \nabla w_k \right\rvert^2 dx \\
				& = \int_{ \Omega } \left( u_k - u_{k-1} \right) u_k dx + \tau^2 \int_{ \Omega } w_k u_k dx + \int_{ \Omega } \cosh \left( w_k \right) \left\lvert \nabla w_k \right\rvert^2 dx \\
				& = \int_{ \Omega } \left( u_k - u_{k-1} \right) u_k dx + \tau^2 \int_{ \Omega } \left( - \Delta u_k + \tau u_k \right) u_k dx + \int_{ \Omega } \cosh \left( w_k \right) \left\lvert \nabla w_k \right\rvert^2 dx \\
				& = \int_{ \Omega } \left( u_k - u_{k-1} \right) u_k dx - \tau^2 \int_{ \Omega } u_k \Delta u_k dx + \tau^3 \int_{ \Omega } dx + \int_{ \Omega } \cosh \left( w_k \right) \left\lvert \nabla w_k \right\rvert^2 dx \\
				& = \int_{ \Omega } \left( u_k - u_{k-1} \right) u_k dx + \tau^2 \int_{ \Omega } \left\lvert \nabla u_k \right\rvert^2 dx + \tau^3 \int_{ \Omega } u^2_k dx + \int_{ \Omega } \cosh \left( w_k \right) \left\lvert \nabla w_k \right\rvert^2 dx \\
				& \geq \int_{ \Omega } \left( u^2_k - u^2_{k-1} \right) dx + \tau^2 \int_{ \Omega } \left\lvert \nabla u_k \right\rvert^2 dx + \tau^3 \int_{ \Omega } u^2_k dx + \int_{ \Omega } \left\lvert \nabla w_k \right\rvert^2 dx. \label{tdp2.4}
			\end{aligned}
		\end{equation}
		Finally for the last integral in equation $ \eqref{tdp2.2} $ we calculate,
		\begin{equation}
			\begin{aligned}
				\tau \int_{ \Omega } \nabla w_k \cdot \nabla u_k dx & = - \tau \int_{ \Omega } w_k \Delta u_k dx \\
				&= - \tau \int_{ \Omega } \left( - \Delta u_k + \tau u_k \right) \Delta u_k dx \\
				& = \tau \int_{ \Omega } \left( \Delta u_k \right)^2 dx - \tau^2 \int_{ \Omega } u_k \Delta u_k dx \\
				&= \tau \int_{ \Omega } \left( \Delta u_k \right)^2 dx + \tau^2 \int_{ \Omega } \left\lvert \nabla u_k \right\rvert^2 dx. \label{tdp2.5}
			\end{aligned}
		\end{equation}
		Upon putting $ \eqref{tdp2.3} $ - $ \eqref{tdp2.5} $ back into equation $ \eqref{tdp2.2} $ we obtain,
		\begin{equation}
			\begin{aligned}
				\frac{1}{\tau} \int_{ \Omega } & \left( \left\lvert \nabla u_k \right\rvert^2 - \left\lvert \nabla u_{k-1} \right\rvert^2 \right) dx + \int_{ \Omega } \left( u^2_k - u^2_{k-1} \right) dx \\
				& + \tau^3 \int_{ \Omega } u^2_k dx + \int_{ \Omega } \left\lvert \nabla w_k \right\rvert^2 dx \\
				& + \tau \int_{ \Omega } \left( \Delta u_k \right)^2 dx +2 \tau^2 \int_{ \Omega } \left\lvert \nabla u_k \right\rvert^2 dx = 0. 
			\end{aligned}
		\end{equation}
		Multiply through this inequality by $ \tau $ and sum up over $k$ to obtain the result.
	\end{proof}
	
	Now we can obtain the discrete version of $ \eqref{apest3} $.
	\begin{prop}\label{p33}
		For the functions given by \eqref{usquig} - \eqref{wbar}, we have the estimate,
		\begin{equation}
			\begin{aligned}
				\int_{ \Omega_T } & \left( \frac{\partial \tilde{w}_j }{\partial t } \right)^2 dx dt + \frac{1}{2} \max_{0 \leq t \leq T } \int_{ \Omega } \left( \frac{ \partial \tilde{u}_j }{ \partial t } \right)^2 dx + \tau^2 \max_{0 \leq t \leq T } \int_{ \Omega } \bar{w}_j \sinh \left( \bar{w}_j \right) dx \\
				& + \frac{1}{2} \max_{0 \leq t \leq T }\tau \int_{ \Omega } \left\lvert \nabla \sinh \left( \bar{w}_j \right) \right\rvert^2 dx + \tau \int_{ \Omega } \left\lvert \nabla \frac{\partial \tilde{u}_j}{\partial t } \right\rvert^2 dx dt + \tau^2 \int_{ \Omega_T} \left( \frac{ \partial \tilde{u}_j }{ \partial t } \right)^2 dx dt \\
				& \leq\int_{ \Omega } \left( \Delta \sinh\left(w_0\right) - \tau w_0 \right)^2 dx + \tau\int_{ \Omega } \left\lvert \nabla \sinh \left( w_0 \right) \right\rvert^2 dx + 2\tau^2 \int_{\Omega } w_0 \sinh \left( w_0 \right) dx. \label{tdp4.1}
			\end{aligned}
		\end{equation}
	\end{prop}
	\begin{proof}
		From equation $ \eqref{td1} $ we derive the equation, for $k=1,2,...$
		\begin{equation}
			\frac{1}{\tau} \left( \frac{u_k - u_{k-1} }{\tau} - \frac{u_{k-1} - u_{k-2} }{\tau} \right) - \Delta \left( \frac{ \sinh \left( w_k \right) - \sinh \left( w_{k-1} \right) }{ \tau } \right) + \tau \frac{ w_k - w_{k-1} }{ \tau } = 0  \;\; \mbox{ in } \; \Omega \label{tdp4.2}
		\end{equation}
		We multiply through this equation by the function $ \frac{ u_k - u_{k-1} }{ \tau } $, and integrate over $ \Omega $ to obtain,
		\begin{equation}
			\begin{aligned}
				\frac{1}{\tau} \int_{\Omega} & \left( \frac{ u_k - u_{k-1} }{ \tau} - \frac{ u_{k-1} - u_{k-2} }{\tau} \right) \left( \frac{u_k - u_{k-1} }{\tau} \right) dx \\
				& - \int_{ \Omega } \Delta \left( \frac{ \sinh \left( w_k \right) - \sinh \left( w_{k-1} \right) }{ \tau } \right) \left( \frac{ u_k - u_{k-1} }{ \tau } \right) dx \\
				& + \tau \int_{ \Omega } \left( \frac{ w_k - w_{k-1} }{\tau} \right) \left( \frac{ u_k - u_{k-1} }{ \tau } \right) dx = 0. \label{tdp4.3}
			\end{aligned}
		\end{equation}
		For the first integral in the above equation we then calculate,
		\begin{equation}
			\begin{aligned}
				\frac{1}{\tau} \int_{ \Omega } & \left( \frac{ u_k - u_{k-1} }{ \tau } - \frac{ u_{k-1} - u_{k-2} }{\tau} \right) \left( \frac{ u_k - u_{k-1} }{ \tau} \right) dx \\
				& \geq \frac{1}{2\tau} \int_{ \Omega } \left[ \left( \frac{ u_k - u_{k-1} }{\tau} \right)^2 - \left( \frac{ u_{k-1} - u_{k-2} }{ \tau } \right)^2 \right] dx. \label{tdp4.4}
			\end{aligned}
		\end{equation}
		For the second integral in $ \eqref{tdp4.3} $ we first calculate,
		\begin{equation}
			\begin{aligned}
				& - \int_{\Omega} \Delta \left( \frac{ \sinh \left( w_k \right) - \sinh \left( w_{k-1} \right) }{ \tau } \right) \left( \frac{ u_k - u_{k-1} }{ \tau } \right) dx \\
				& = - \int_{ \Omega } \left( \frac{ \sinh \left( w_k \right) - \sinh \left( w_{k-1} \right) }{ \tau } \right) \Delta \left( \frac{ u_k - u_{k-1} }{\tau} \right) dx \\
				& = \int_{ \Omega } \left( \frac{ \sinh \left( w_k \right) - \sinh \left( w_{k-1} \right) }{ \tau } \right) \left( - \tau \left( \frac{ u_k - u_{k-1} }{\tau } \right) + \left( \frac{ w_k - w_{k-1} }{ \tau } \right) \right) dx \\
				& = - \tau \int_{ \Omega } \left( \frac{ \sinh \left( w_k \right) - \sinh \left( w_{k-1} \right) }{ \tau } \right) \left( \frac{ u_k - u_{k-1} }{ \tau } \right) dx \\
				& + \int_{\Omega } \left( \frac{ \sinh \left( w_k \right) - \sinh \left( w_{k-1} \right) }{ \tau } \right) \left( \frac{w_k - w_{k-1} }{ \tau } \right) dx \\
				& = \int_{ \Omega } \left( \sinh \left( w_k \right) - \sinh \left( w_{k-1} \right) \right) \left( - \Delta \sinh \left( w_k \right) + \tau w_k \right) dx \\
				& + \int_{ \Omega } \left( \frac{ \sinh \left( w_k \right) - \sinh \left( w_{k-1} \right) }{ \tau } \right)  \left( \frac{w_k - w_{k-1} }{ \tau } \right) dx \\
				& = - \int_{ \Omega } \left( \sinh \left( w_k \right) - \sinh \left( w_{k-1} \right) \right) \Delta \sinh \left( w_k \right) dx + \tau \int_{ \Omega } \left( \sinh \left( w_k \right) - \sinh \left( w_{k-1} \right) \right) w_k dx \\
				& + \int_{ \Omega } \left( \frac{ \sinh \left( w_k \right) - \sinh \left( w_{k-1} \right) }{ \tau } \right)  \left( \frac{w_k - w_{k-1} }{ \tau } \right) dx \label{tdp4.5}
			\end{aligned}
		\end{equation}
		We now have three integrals to worry about here. For the first integral on the right most side of this equation we have,
		\begin{equation}
			\begin{aligned}
				- \int_{ \Omega } &\left( \sinh \left( w_k \right) - \sinh \left( w_{k-1} \right) \right) \Delta \sinh \left( w_k \right) dx \\
				& = \int_{ \Omega } \left( \nabla \sinh \left( w_k \right) - \nabla \sinh \left( w_{k-1} \right) \right) \nabla \sinh \left( w_k \right) dx \\
				& \geq \frac{1}{2} \int_{ \Omega } \left( \left\lvert \nabla \sinh \left( w_k \right) \right\rvert^2 - \left\lvert \nabla \sinh \left( w_{k-1} \right) \right\rvert^2 \right) dx \label{tdp4.6}
			\end{aligned}
		\end{equation}
		Towards estimating the second integral we first calculate,
		\begin{equation}
			\begin{aligned}
				\left( \sinh \left( w_k \right) - \sinh \left( w_{k-1} \right) \right) w_k & = \left[ \frac{1}{2} \left( e^{w_k} - e^{-w_k} \right) - \frac{1}{2} \left( e^{w_{k-1}} - e^{- w_{k-1}} \right) \right] w_k \\
				& = \left[ \frac{1}{2} \left( e^{w_k} - e^{w_{k-1}} \right) + \frac{1}{2} \left( e^{- w_{k-1}} - e^{-w_k} \right) \right] w_k \\
				& = \frac{1}{2} \left( e^{w_k} - e^{w_{k-1}} \right) w_k + \frac{1}{2} \left( e^{-w_k} - e^{- w_{k-1}} \right) \left( - w_k \right) \\
				& \geq \frac{1}{2} \left( w_k e^{w_k} - w_{k-1} e^{w_{k-1}} - \left( w_k - w_{k-1} \right) \right) \\
				&+ \frac{1}{2} \left( - w_k e^{-w_k} + w_{k-1} e^{- w_{k-1}} - \left( - w_k + w_{k-1} \right) \right) \\
				& = w_k \frac{1}{2} \left( e^{w_k} - e^{ -w_k} \right) - w_{k-1} \frac{1}{2} \left( e^{w_{k-1}} - e^{-w_{k-1}} \right) \\
				& = w_k \sinh \left( w_k \right) - w_{k-1} \sinh \left( w_{k-1} \right) 
			\end{aligned}
		\end{equation}
		Consequently we then have,
		\begin{equation}
			\tau \int_{ \Omega } \left( \sinh \left( w_k \right) - \sinh \left( w_{k-1} \right) \right) w_k dx \geq \tau \int_{ \Omega } \left( w_k \sinh \left( w_k \right) - w_{k-1} \sinh \left( w_{k-1} \right) \right) dx. \label{tdp4.7}
		\end{equation}
		For the last integral in equation $ \eqref{tdp4.5} $ we then calculate using the mean value theorem.
		\begin{equation}
			\begin{aligned}
				\int_{ \Omega } \left( \frac{ \sinh \left( w_k \right) - \sinh \left( w_{k-1} \right) }{ \tau } \right)& \left( \frac{ w_k - w_{k-1} }{ \tau } \right) dx  = \int_{ \Omega } \cosh \left( \xi \right) \left( \frac{ w_k - w_{k-1} }{ \tau } \right)^2 dx \\
				& \geq \int_{ \Omega } \left( \frac{ w_k - w_{k-1} }{ \tau } \right)^2 dx
			\end{aligned}
		\end{equation}
		Using the above inequality and $ \eqref{tdp4.6} $ - $ \eqref{tdp4.7} $ in equation $ \eqref{tdp4.5} $ we are then able to obtain,
		\begin{equation}
			\begin{aligned}
				- \int_{ \Omega } \Delta & \left( \frac{ \sinh \left( w_k \right) - \sinh \left( w_{k-1} \right) }{ \tau } \right) \left( \frac{ u_k - u_{k-1} }{ \tau } \right) dx \\
				&\geq \frac{1}{2} \int_{ \Omega } \left( \left\lvert \nabla \sinh \left( w_k \right) \right\rvert^2 - \left\lvert \nabla \sinh \left( w_{k-1} \right) \right\rvert^2 \right) dx \\
				& + \tau \int_{ \Omega } \left( w_k \sinh \left( w_k \right) - w_{k-1} \sinh \left( w_{k-1} \right) \right) dx + \int_{ \Omega } \left( \frac{ w_k - w_{k-1} }{ \tau } \right)^2 dx \label{tdp4.8}
			\end{aligned}
		\end{equation}
		Finally for the last integral in equation $ \eqref{tdp4.3} $ we have,
		\begin{equation}
			\begin{aligned}
				\tau \int_{ \Omega } \left( \frac{ w_k - w_{k-1} }{ \tau } \right) & \left( \frac{ u_k - u_{k-1} }{ \tau } \right) dx = - \tau \int_{ \Omega } \Delta \left( \frac{ u_k - u_{k-1} }{ \tau } \right) \left( \frac{ u_k - u_{k-1} }{ \tau } \right) dx \\
				& + \tau^2 \int_{ \Omega } \left( \frac{ u_k - u_{k-1} }{ \tau } \right)^2 dx \\
				& = \tau \int_{ \Omega } \left\lvert \nabla \left( \frac{ u_k - u_{k-1} }{ \tau } \right) \right\rvert^2 dx + \tau^2 \int_{ \Omega } \left( \frac{ u_k - u_{k-1} }{ \tau } \right)^2 dx \label{tdp4.9}
			\end{aligned}
		\end{equation}
		Then using $ \eqref{tdp4.4}, \eqref{tdp4.8}, $ and $ \eqref{tdp4.9} $, in equation $ \eqref{tdp4.3} $, we obtain the inequality, 
		\begin{equation}
			\begin{aligned}
				\frac{1}{2\tau} \int_{ \Omega } & \left[ \left( \frac{ u_k - u_{k-1} }{\tau} \right)^2 - \left( \frac{ u_{k-1} - u_{k-2} }{ \tau } \right)^2 \right] dx + \int_{ \Omega } \left( \frac{ w_k - w_{k-1} }{ \tau } \right)^2 dx \\
				& + \frac{1}{2} \int_{ \Omega } \left( \left\lvert \nabla \sinh \left( w_k \right) \right\rvert^2 - \left\lvert \nabla \sinh \left( w_{k-1} \right) \right\rvert^2 \right) dx + \tau \int_{ \Omega } \left\lvert \nabla \left( \frac{u_k - u_{k-1}}{\tau} \right) \right\rvert^2 dx \\
				& + \tau \int_{ \Omega } \left( w_k \sinh \left( w_k \right) - w_{k-1} \sinh \left( w_{k-1} \right) \right) dx + \tau^2 \int_{ \Omega } \left( \frac{ u_k - u_{k-1} }{ \tau } \right)^2 dx \leq 0.
			\end{aligned}
		\end{equation}
		Multiply through this inequality by $ \tau $, sum up over $ k $, and take into account \eqref{wu1} to get the result.
	\end{proof}
	
	We are now ready to prove the necessary compactness to justify taking the limit in our equations. First we will prove the compactness for the piecewise linear approximations $\tilde{u}_j$ and $\tilde{w}_j$ defined by \eqref{usquig} and \eqref{wsquig}. 
	\begin{prop}
		The sequences $ \{\tilde{u}_j \} $ and $ \{ \tilde{w}_j \} $ are bounded in $W^{1,2}(\Omega_T)$, and hence precompact in $ L^2 \left( \Omega_T \right) $.
	\end{prop}
	\begin{proof}
		To demonstrate this we first notice that by Proposition \ref{p31} the sequence $ \{ \frac{\partial \tilde{u}_j }{\partial t} \} $ is bounded in $ L^2 \left( \Omega_T \right) $. For each $t\in(0,T]$ there is a $k$ such that $t\in (t_{k-1}, t_k]$. Then we have 
		\begin{equation}
			\begin{aligned}
				\int_{ \Omega } \left\lvert \nabla \tilde{u}_j \left( x, t \right) \right\rvert^2 dxdt & =   \int_{ \Omega } \left\lvert \frac{t-t_{k-1}}{\tau} \nabla u_k + \left( 1 - \frac{t-t_{k-1}}{\tau} \right) \nabla u_{k-1} \right\rvert^2 dxdt \\
				& \leq   \frac{t-t_k}{\tau} \int_{ \Omega } \left\lvert \nabla u_k \right\rvert^2 dx + \left( 1 - \frac{t-t_k}{\tau} \right) \int_{ \Omega } \left\lvert \nabla u_{k-1} \right\rvert^2 dx  \\
				&\leq \sup_{ 0 \leq t \leq T }\int_{ \Omega } \left\lvert \nabla \bar{u}_j \right\rvert^2 dx\leq c.\label{rs1}
			\end{aligned}
		\end{equation}
		The last step is due to Proposition \ref{p32}. Now we integrate \eqref{ttd1} over $\Omega$ to get
		$$\frac{d}{dt}\int_{ \Omega }\tilde{u}_j \left( x, t \right)dx+\tau\int_{ \Omega }\bar{w}_j(x,t)dx=0.$$
		Integrate with respect to $t$ and keep \eqref{tdp1.1} in mind to deduce
		\begin{equation}
			\left|\int_{ \Omega }\tilde{u}_j \left( x, t \right)dx\right|\leq c.
		\end{equation}
		By Poincar\'{e}'s inequality, we have
		\begin{eqnarray}
			\int_{ \Omega }\tilde{u}_j^2 \left( x, t \right)dx&\leq &2\int_{ \Omega }\left(\tilde{u}_j \left( x, t \right)-\frac{1}{|\Omega|}\int_{ \Omega }\tilde{u}_j \left( x, t \right)dx\right)^2dx+\frac{2}{|\Omega|}\left(\int_{ \Omega }\tilde{u}_j \left( x, t \right)dx\right)^2\nonumber\\
			&\leq&c\int_{ \Omega }\left|\nabla \tilde{u}_j \left( x, t \right)\right|^2dx+c\leq c.
		\end{eqnarray}
		Now we can conclude that the sequence $ \{ \tilde{u}_j \} $ is bounded in $ W^{1,2} \left( \Omega_T \right) $. As such we may conclude that the sequence is precompact in $ L^2 \left( \Omega_T \right) $.
		In the same manner, only using Propositions \ref{p32} and \ref{p33},  we also have that the sequence $ \{ \tilde{w}_j \} $ is bounded in $ W^{1,2} \left( \Omega_T \right) $. The main difference is that instead of \eqref{rs1}
		we use the estimate	\begin{equation}
			\begin{aligned}
				\int_{ \Omega_T } \left\lvert \nabla \tilde{w}_j \left( x, t \right) \right\rvert^2 dxdt & = \sum_{k=1}^{j} \int_{t_{k-1}}^{t_{k}} \int_{ \Omega } \left\lvert \frac{t-t_{k-1}}{\tau} \nabla w_k + \left( 1 - \frac{t-t_{k-1}}{\tau} \right) \nabla w_{k-1} \right\rvert^2 dxdt \\
				& \leq \sum_{k=1}^{j} \int_{t_{k-1}}^{t_{k}} \left[ \frac{t-t_{k-1}}{\tau} \int_{ \Omega } \left\lvert \nabla w_k \right\rvert^2 dx + \left( 1 - \frac{t-t_{k-1}}{\tau} \right) \int_{ \Omega } \left\lvert \nabla w_{k-1} \right\rvert^2 dx \right] dt \\
				& \leq \sum_{k=1}^{j} \tau \left[ \int_{ \Omega } \left\lvert \nabla w_k \right\rvert^2 dx + \int_{ \Omega } \left\lvert \nabla w_{k-1} \right\rvert^2 dx \right] \\
				& \leq c \int_{ \Omega_T } \left\lvert \nabla \bar{w}_j \right\rvert^2 dx dt + \tau c \int_{ \Omega } \left\lvert \nabla w_0 \right\rvert^2 dx \leq c\nonumber
			\end{aligned}
		\end{equation}
		and
		$$	\left|\int_{ \Omega }\tilde{w}_j \left( x, t \right)dx\right|\leq c$$
		is a consequence of \eqref{ttd2}. The proof is complete.
	\end{proof}

	In the preceding proposition, the linear approximations $\tilde{u}$ and $\tilde{w}$ were crucial in order to be able to justify the regularity with respect to time. Our next proposition deals with the compactness of the piecewise constant functions $\bar{u}$ and $\bar{w}$. 

	\begin{prop}
		The sequences $ \{ \bar{u}_j \} $ and $ \{ \bar{w}_j \} $ given by \eqref{ubar} and \eqref{wbar} are precompact in $ L^2 \left( \Omega_T \right) $.
	\end{prop}
	\begin{proof}
		To go about proving this proposition we calculate using \eqref{usquig} and \eqref{ubar}, for $ t \in (t_{k-1},t_k] $,
		\begin{equation}
			\begin{aligned}
				\tilde{u}_j \left( x, t \right) - \bar{u}_j \left( x, t \right) & = \frac{t-t_k}{\tau} \left( u_k - u_{k-1} \right) \\
				& = \left( t - t_k \right) \frac{ \partial \tilde{u}_j }{ \partial t }. 
			\end{aligned}
		\end{equation}
		Consequently,	
		\begin{eqnarray}
			\int_{ \Omega_T } \left( \tilde{u}_j - \bar{u}_j \right)^2 dx dt&=& \sum_{k=1}^{j}  \int_{ \Omega }\int_{t_{k-1}}^{t_k}\left(\left( t - t_k \right) \frac{ \partial \tilde{u}_j }{ \partial t }\right)^2dtdx \nonumber\\
			&\leq& \frac{\tau^2}{3} \int_{ \Omega_T } \left( \frac{ \partial \tilde{u}_j }{ \partial t} \right)^2 dx dt\leq c\tau^2.\label{tb1}
		\end{eqnarray}
		This implies that the sequence $ \{ \bar{u}_j \} $ is precompact in $ L^2 \left( \Omega_T \right) $. 
		
		The  estimate \eqref{tb1} also holds for $ \left( \tilde{w}_j - \bar{w}_j \right) $. The only difference in the proof is that we use \eqref{wsquig} and \eqref{wbar}. Subsequently we may conclude that the sequence $ \{ \bar{w}_j \} $ is precompact in $ L^2 \left( \Omega_T \right) $. This completes the proof.
	\end{proof}
	\begin{prop}\label{p36}
		The sequences $ \{\sinh\left(\bar{w}_j \right)\} $ and $ \{\bar{u}_j \} $ are bounded in $L^2\left( 0,T; W^{2,2} \left( \Omega \right)\right)  $. 
	\end{prop}
	\begin{proof} We can easily infer from \eqref{ttd2} and Lemma \ref{lwt} that $ \{\bar{u}_j \} $ is bounded in $L^2\left( 0,T; W^{2,2} \left( \Omega \right)\right)  $. 
		To see the rest,  we clearly have for $ s \geq 0 $ the following inequality,
		\begin{equation}
			\sinh(s) \leq \cosh(s).
		\end{equation}
		Then since $ \sinh $ is odd and $ \cosh $ is even we able to conclude that,
		\begin{equation}
			\left\lvert \sinh(s) \right\rvert \leq \cosh(s).
		\end{equation}
		Then using Proposition \ref{p31}, we get the estimate,
		\begin{equation}
			\int_{ \Omega } \left\lvert \sinh(\bar{w}_j) \right\rvert dx \leq \int_{ \Omega } \cosh \left( \bar{w}_j \right) dx \leq c.
		\end{equation}
		We calculate from Lemmas \ref{l22} and \ref{lwt} that 
		\begin{equation}
			\begin{aligned}
				\left\lVert \nabla\sinh \left( \bar{w}_j \right) \right\rVert_{2,\Omega} &\leq \epsilon\left\lVert \nabla^2 \sinh \left( \bar{w}_j \right) \right\rVert_{2,\Omega} + c \left\lVert \sinh \left( \bar{w}_j \right) \right\rVert_{1,\Omega} \\
				& \leq c\epsilon\left\lVert\Delta \sinh \left( \bar{w}_j \right) \right\rVert_{2,\Omega} +c\epsilon\left\lVert \nabla\sinh \left( \bar{w}_j \right) \right\rVert_{2,\Omega} + c .
			\end{aligned}
		\end{equation}
		Choose $\epsilon$ suitably small to get
		$$\left\lVert \nabla\sinh \left( \bar{w}_j \right) \right\rVert_{2,\Omega}\leq c\left\lVert\Delta \sinh \left( \bar{w}_j \right) \right\rVert_{2,\Omega}+c.$$
		Square the above inequality and then integrate to get
		$$\iot\left| \nabla\sinh \left( \bar{w}_j \right)\right|^2dxdt\leq c\iot\left| \Delta\sinh \left( \bar{w}_j \right)\right|^2dxdt\leq c.$$
		The last step is due to Proposition \ref{p31}. Use Lemma \ref{l22} again to get
		$$\iot\left|\sinh \left( \bar{w}_j \right)\right|^2dxdt\leq c.$$
		Invoking Proposition \ref{p31} and Lemma \ref{lwt} one more time, we can derive
		$$\iot\left| \nabla^2\sinh \left( \bar{w}_j \right)\right|^2dxdt\leq c\iot\left| \Delta\sinh \left( \bar{w}_j \right)\right|^2dxdt+ c\iot\left| \nabla\sinh \left( \bar{w}_j \right)\right|^2dxdt\leq c.$$
		This finishes the proof.
	\end{proof}
	
	We are now ready to prove our main theorem.
	
	\begin{proof}[Proof of Main Theorem]
		Passing to subsequences if need be we may assume,
		\begin{equation}
			\begin{aligned}
				\bar{u}_j \rightarrow &u \;\; \mbox{ weakly in } \; L^2(0,T;\wot), \; \mbox{ strongly in } \; L^2 \left( \Omega_T \right), \; \mbox{ and a.e. on } \; \Omega_T,\\
				\bar{w}_j \rightarrow &w \;\; \mbox{ weakly in } \; L^2(0,T;\wot) , \; \mbox{ strongly in } \; L^2 \left( \Omega_T \right), \; \mbox{ and a.e. on } \; \Omega_T.
			\end{aligned}
		\end{equation}
		By \eqref{tb1}, we also have
		$$	\tilde{u}_j \rightarrow u \;\; \mbox{ weakly in } \; W^{1,2}(\Omega_T).$$
		On account of Proposition \ref{p36} we obtain
		\begin{equation}
			\sinh \left( \bar{w}_j \right) \rightarrow \sinh \left( w \right) \;\; \mbox{ weakly in } \; L^2 \left( 0, T; W^{2,2} \left( \Omega \right) \right) \;.
		\end{equation}
		Thus we may pass to the limit in \eqref{ttd1} and \eqref{ttd2}. The proof is complete.
	\end{proof}
	

	


\end{document}